\documentclass[11pt]{amsart}

\usepackage{setspace}
\usepackage{amssymb, amsthm, amscd}
\usepackage{latexsym}
\usepackage{amssymb}
\usepackage[T1]{fontenc}	
\usepackage[english]{babel}
\usepackage{amsmath,amssymb}
\usepackage{amsfonts}
\usepackage{amscd}
\usepackage{cite}
\usepackage{marvosym}
\usepackage{mathrsfs}
\usepackage{amsthm}

\newtheorem{theorem}{Theorem}[section]
\newtheorem*{thm}{Theorem}

\newtheorem{lemma}[theorem]{Lemma}

\newtheorem*{problemA}{Problem (A)}
\newtheorem*{problemB}{Problem (B)}
\newtheorem*{theorem1}{Theorem 1}
\newtheorem*{theorem2}{Theorem 2}
\newtheorem*{corollary3}{Corollary 3}

\newtheorem*{definition4}{Definition 4}
\newtheorem*{definition5}{Definition 5}
\newtheorem*{definition6}{Definition 6}
\newtheorem*{definition7}{Definition 7}
\newtheorem*{definition8}{Definition 8}
\newtheorem*{theorem9}{Theorem 9}
\newtheorem*{lemma10}{Lemma 10}
\newtheorem*{lemma11}{Lemma 11}
\newtheorem*{lemma12}{Lemma 12}
\newtheorem*{lemma13}{Lemma 13}
\newtheorem*{lemma14}{Lemma 14}
\newtheorem*{lemma15}{Lemma 15}
\newtheorem*{definition16}{Definition 16}
\newtheorem*{lemma17}{Lemma 17}
\newtheorem*{lemma18}{Lemma 18}
\newtheorem*{lemma19}{Lemma 19}

\theoremstyle{definition}

\theoremstyle{remark}

\numberwithin{equation}{section}

\newcommand{\R}{\mathbb{R}}

\newcommand{\e}{e_\alpha}

\newcommand{\be}{\begin{equation}}
\newcommand{\ee}{\end{equation}}
\newcommand{\bd}{\begin{displaymath}}
\newcommand{\ed}{\end{displaymath}}
\newcommand{\del}{\delta}

\renewcommand{\d}{\delta }

\renewcommand{\e}{\varepsilon}

\renewcommand{\d}{\delta }

\renewcommand{\l }{\lambda }

\begin{document}
	\date{}
	\title[Rigidity of stable minimal hypersurfaces in asymptotically flat spaces]
	{Rigidity of stable minimal hypersurfaces in \\  asymptotically flat spaces}
	\author{Alessandro Carlotto}
	\address{ETH - Institute for Theoretical Studies  \\
		ETH \\
		Z\"urich, Switzerland}
	\email{alessandro.carlotto@eth-its.ethz.ch}
	
	\thanks{During the preparation of this work, the author was partially supported by Stanford University and NSF grant DMS-1105323.}

	\begin{abstract} We prove that if an asymptotically Schwarzschildean 3-manifold $(M,g)$ contains a properly embedded stable minimal surface, then it is isometric to the Euclidean space. This implies, for instance, that in presence of a positive ADM mass any sequence of solutions to the Plateau problem with diverging boundaries can never have uniform height bounds, even at a single point. An analogous result holds true up to ambient dimension seven provided polynomial volume growth on the hypersurface is assumed.
	\end{abstract}
	
	\maketitle
	\section{Introduction}
	
	Asymptotically flat manifolds naturally arise, in general relativity, as models for \textsl{isolated gravitational systems} and can be regarded as one of the most basic classes of solutions to the Einstein equations. Their study has flourished over the last fifty years and the geometric and physical properties of these spaces have been widely investigated. In this article, our work is centered around the following fundamental question:
	
	\begin{problemA}\label{question}
		Are there complete, stable minimal hypersurfaces in asymptotically flat manifolds?
	\end{problemA}
	
	Throughout this paper, the word \textsl{complete} is always meant to implicitly refer to \textsl{non compact} minimal hypersurfaces without boundary. Furthermore, we shall tacitly assume all our hypersurfaces to be two-sided.
	
	Apart from their intrinsic relevance, stable minimal hypersurfaces naturally arise as limits of two categories of important variational objects:
	\begin{enumerate}
		\item[(L1)]{sequences of minimizing currents solving the Plateau problem for diverging boundaries}
		\item[(L2)]{sequences of large isoperimetric boundaries or, more generally, of large volume-preserving stable CMC hypersurfaces \cite{EM12, EM13}}
	\end{enumerate}
	and thus their study plays a key role in the process of deeper understanding the large-scale geometry of initial data sets.
	
	From our perspective, the previous question had a twofold motivation: on the one end it could be regarded as a natural extension of the analysis of stable minimal hypersurfaces in the Euclidean space, on the other it implicitly arose in the proof of the \textsl{Positive Mass Theorem} by Schoen-Yau \cite{SY79}. Even in the most basic of all cases, namely when $(M,g)$ is $\mathbb{R}^{n}$ with its flat metric and $\Sigma^{n-1}$ is assumed to be an entire minimal graph, the study of Problem (A) has played a crucial role in the development of Analysis along the whole course of the twentieth century:
	
	\begin{problemB}\label{bernstein}
		Are affine functions the only entire minimal graphs over $\mathbb{R}^{n-1}$ in $\mathbb{R}^{n}$?
	\end{problemB}
	
	Indeed, minimal graphs are automatically stable (in fact: area-minimizing) by virtue of a well-known calibration argument \cite{CM11} and thus Problem (B) can be regarded as the most special subcase of Problem (A).
	Such problem, which is typically named after S. N. Bernstein, was formulated around 1917 \cite{Ber17} as an extension of the $n=3$ case, which Bernstein himself had settled (see also \cite{Fle62} for a different approach). In higher dimensions, the answer is positive only up to ambient dimension 8 and is due to De Giorgi (for $n=4$, \cite{dG65}), Almgren (for $n=5$, \cite{Alm66}) and Simons (for $6\leq n\leq 8$, \cite{Sim68}) who also showed that the conjecture is false for $n\geq 9$ because of the existence of non-trivial area-minimizing cones in $\mathbb{R}^{n-1}$ (see \cite{BdGG69}).
	
	When the ambient manifold is Euclidean, but $\Sigma$ is only known to be stable (and not necessarily graphical) a similar classification result is only known when $n=3$ and it was obtained independently by do Carmo and Peng \cite{dCP79} and Fischer-Colbrie and Schoen \cite{FS80}:
	
	\begin{thm}\cite{dCP79,FS80}
		The only complete stable oriented minimal surface in $\R^{3}$ is the plane.
	\end{thm} 
	
	However, the same statement is still not known to be true in $\mathbb{R}^{n}$ for $n\geq 4$ unless the minimal hypersurface $\Sigma^{n-1}$ under consideration is assumed to have \textsl{polynomial volume growth} meaning that for some (hence for any) point $p$
	\[
	\mathscr{H}^{n-1}(\Sigma\cap B_{r}(p))\leq \theta^{\ast} r^{n-1}, \ \ \textrm{for all} \ r>0.
	\]

	Before stating our main theorems concerning question (A), we need to recall an essential physical assumption which will always be tacitly made in the sequel of this article. It is customary in general relativity to assume that the energy density measured by any physical observer is \textsl{non-negative} at each point: this turns out to imply the requirement, in the time-symmetric case (which is the one we are considering here), that the scalar curvature be non-negative at all points of $M$.
	
	In all of the statements below and throughout the paper we shall assume that the manifold $(M,g)$ is connected, orientable and has finitely many ends.
	
	Our first theorem states that there is a wide and phyisically relevant class of asymptotically flat manifolds for which positivity of the ADM mass is an obstruction to the existence of stable minimal surfaces. Once again, we refer the readers to the next section for the defintions of ADM mass: for the sake of this Introduction, they may consider the ADM mass $\mathcal{M}$ a scalar quantity measuring the gravitational deformation of $(M,g)$ from the trivial couple $(\R^{n},\delta)$.
	
	\begin{theorem1}\label{thm:1}
		Let $(M,g)$ be an asymptotically Schwarzschildean 3-manifold of non-negative scalar curvature. If it contains a complete, properly embedded stable minimal surface $\Sigma$, then $\left(M,g\right)$ is isometric to the Euclidean space $\R^{3}$ and $\Sigma$ is an affine plane. 
	\end{theorem1}  
	
	An analogous result is obtained for ambient dimension $4\leq n<8$ under an a-priori bound on the volume growth of $\Sigma$ (as we specified above) and provided the stability assumption is replaced by \textsl{strong} stability.

	\begin{theorem2} \label{thm:2}
		Let $(M,g)$ be an asymptotically Schwarzschildean manifold of dimension $4\leq n<8$ and non-negative scalar curvature. If it contains a complete, properly embedded strongly stable minimal hypersurface $\Sigma$ of polynomial volume growth, then $\left(M,g\right)$ is isometric to the Euclidean space $\R^{n}$ and $\Sigma$ is an affine hyperplane.
	\end{theorem2}
	
	These two theorems also apply to the physically relevant case when the ambient manifold $M$ has a compact boundary (an \textsl{horizon}, for instance) once it is assumed that $\Sigma\subset M\setminus\partial M$. If instead $\Sigma$ is allowed to intersect the boundary $\partial M$ (and thus to have a boundary $\partial\Sigma$), then a variation on the argument we are about to present allows to prove that there cannot be stable \textsl{free boundary} minimal surfaces in non-trivial asymptotically Schwarzschildean 3-manifolds with weakly mean-convex boundary $\partial M$ (with respect to the outward pointing unit normal).
	
	We expect Theorem 1 to be sharp at that level of generality, and more specifically we certainly do not expect the assumption of \textsl{properness} to be inessential to the above theorems, unless $\Sigma$ is assumed to be (locally) area-minimizing in which case properness can be easily proved via a standard local replacement argument. Moreover, we expect the minimal surface $\Sigma$ to be automatically proper under the additional assumption that $M$ has an \textsl{horizon boundary}, more precisely that $\partial M$ is a finite union of minimal spheres and there are no other closed minimal surfaces in $M$. However, we will not investigate these aspects further in this work as we plan to analyse them carefully in a forthcoming paper with O. Chodosh and M. Eichmair.
	
	Theorem 1 has several remarkable consequences and, among these, we would like to mention an application to the study of sequences of solutions to the Plateau problem for diverging boundaries belonging to a given hypersurface. We will first need some terminology.
	Given an asymptotically flat manifold $(M,g)$ \textsl{with one end} and, correspondingly, a system of asymptotically flat coordinates $\left\{x\right\}$ we call \textsl{hyperplane} a subset of the form $\Pi=\left\{x\in\R^{n}\setminus B \ | \sum_{i=1}^{n}a_{i}x^{i}=0\right\}$ for some real numbers $a_{1},a_{2},\ldots, a_{n}$. Possibly by changing $\left\{x\right\}$ (we do not rename) one can always reduce to the case when $a_{1}=\ldots=a_{n-1}=0$ and $a_{n}=1$. In this setting, we define \textsl{height} of a point in $M\setminus Z \simeq \R^{n}\setminus B$ the value of its $x^{n}-$coordinate. 
	Moreover, we denote by $x'$ the ordered $(n-1)-$tuple corresponding to the first $(n-1)$ coordinates of a point in $M\setminus Z$.
	
	\begin{corollary3}\label{cor:3} Let $(M,g)$ be an asymptotically Schwarzschildean 3-manifold of non-negative scalar curvature and positive ADM mass, let $\Pi$ an hyperplane and let $\left(\Omega_{i}\right)_{i\in\mathbb{N}}$ any monotonically increasing sequence of regular, relatively compact domains such that $\cup_{i}\Omega_{i}=\Pi$. For any index $i$, define $\Gamma_{i}$ to be a solution of the Plateau problem with boundary $\partial\Omega_{i}$. Then for each $x'\in\Pi$ the sequence $\left(\Gamma_{i}\right)_{i\in\mathbb{N}}$ cannot have uniformly bounded height at $x'$, namely
		\[
		\liminf_{i\to\infty} \min_{\left(x',x^{3}\right)\in\Gamma_{i}} |x^{3}|=+\infty.
		\]
	\end{corollary3}
	
	This corollary follows at once from Theorem 1 by means of a standard compactness argument (see, for instance, \cite{Whi87b}).
	
	As anticipated above, the proofs of our rigidity results are inspired by the proof given by Schoen-Yau of the \textsl{Positive Mass Theorem} in \cite{SY79} where (arguing by contradiction) negativity of the ADM mass is exploited for constructing a (strongly) stable complete minimal surface of planar type, thereby violating the stability inequality by a preliminary reduction, via a density argument, to a Riemannian metric of strictly positive scalar curvature, at least outside a compact set. In our case, we need to deal with two substantial differences:
	\begin{enumerate}
		\item{the structure and the behaviour at infinity of the hypersurface $\Sigma$  (in terms of topology, number of ends, asymptotics) are \textsl{not} known a priori;}
		\item{the metric $g$ is only required to have non-negative scalar curvature, thereby admitting the (relevant) case when it is in fact scalar flat, as prescribed by the Einstein constraints in the \textsl{vacuum} case.}
	\end{enumerate}
	
	These delicate aspects are not dealt with in previous works (specifically: the statements about large CMC spheres in \cite{EM12} either assume \textsl{quadratic area growth} or \textsl{strict positivity} of the ambient scalar curvature) and indeed one crucial part of our study (and a preliminary step in the proof of Theorem 1 and Theorem 2) is to characterize the structure at infinity of a complete minimal hypersurface having finite Morse index.

	Essentially, we extend to asymptotically flat manifolds the Euclidean structure theorem by Schoen \cite{Sch83} which states, roughly speaking, that any such minimal hypersurface has to be \textsl{regular at infinity} in the sense that it can be decomposed (outside a compact set) as a finite union of graphs with at most logarithmic growth when $n=3$ or polynomial decay (like $|x'|^{3-n}$) when $n\geq 4$. Schoen proved this theorem making substantial use of the Weierstrass representation for minimal surfaces, a tool which is strongly peculiar of the Euclidean setting as its applicability relies on the fact that the coordinate function have harmonic restrictions to minimal submanifolds. Instead, our approach has a more analytic character and is thus applicable to a wider class of spaces.

	The results contained in this article have been first announced in June 2013 and made available, in a preliminary form, in October 2013: in that version we gave a rather different proof of Theorem 1 which seemed to require a quadratic area growth assumption, in analogy with the higher dimensional case. While completing the preparation of the present article, we were communicated that O. Chodosh and M. Eichmair were able to remove, independently of us, such assumption in our rigidity theorem. 
	
	We would like to conclude this Introduction by mentioning the very recent paper \cite{CS14} by the author and R. Schoen, where we show that the rigidity Theorem 1 (as well as Theorem 2) is essentially sharp by constructing asymptotically flat solutions of the Einstein constraint equations in $\R^{n}$ (for $n\geq 3$) that have positive ADM mass and are \textsl{exactly flat} outside of a solid cone (for any positive value of the corresponding opening angle) so that they contain plenty of complete, stable hypersurfaces. A posteriori, this strongly justifies our requirement that the metric $g$ is asymptotically Schwarzschildean. 
	
	\
	
	\section{Definitions and notations}

	We need to start by recalling the definition of weighted Sobolev and H\"older spaces in the Euclidean setting $(\R^{n},\delta)$. Here and in the sequel we always assume $n\geq 3$ and we let $r\in \mathcal{C}^{\infty}(\R^{n})$ any positive function which equals the usual Euclidean distance $|\cdot|$ outside of the unit ball. 
	
	\begin{definition4}\label{def:SobSpace} Given an index $1\leq p\leq \infty$ and a weight $\delta\in\R$ we define the weighted Lebesgue spaces $\mathcal{L}^{p}_{\delta}(\R^{n})$ as the sets of measurable functions on $\R^{n}$ such that the corresponding norms $\left\|\cdot\right\|_{\mathcal{L}^{p}_{\delta}}$ are finite, with
		\[
		\left\|u\right\|_{\mathcal{L}^{p}_{\delta}}=
		\begin{cases} \left(\int_{\R^{n}}|u|^{p}r^{-\delta p-n}\,d\mathscr{L}^{n}\right)^{1/p}, \ & \ \textrm{if} \ p<\infty \\
		\left\|r^{-\delta}u\right\|_{\mathcal{L}^{\infty}}, \ & \ \textrm{if} \ p=\infty.
		\end{cases}
		\]
		Correspondingly, for an integer $k\in\mathbb{N}$ we define the weighted Sobolev spaces $\mathcal{W}^{k,p}_{\delta}(\R^{n})$ as the sets of measurable functions for which the norms
		\[ \left\|u\right\|_{\mathcal{W}^{k,p}_{\delta}}=\sum_{0\leq|\gamma|\leq k} \left\|\partial^{\gamma}u\right\|_{\mathcal{L}^{p}_{\delta-|\gamma|}}
		\]
		are finite.
	\end{definition4}
	
	\begin{definition5}\label{def:HolSpace}
		Given an integer $k\in\mathbb{N}$, and real numbers $\delta\in\R$ and $\beta\in\left(0,1\right)$ we define the H\"older space $\mathcal{C}^{k,\beta}_{\delta}(\mathbb{R}^{n})$ as the set of continuous functions on $\R^{n}$ for which the norm $\left\|\cdot\right\|_{\mathcal{C}^{k,\beta}_{\delta}}$
		\[ \left\|u\right\|_{\mathcal{C}^{k,\beta}_{\delta}}=\sum_{0\leq|\gamma|\leq k} \sup_{x\in\R^{n}} r(x)^{-\delta+|\gamma|}\left|\partial^{\gamma}u(x)\right|+\sum_{|\gamma|=k}\sup_{|x-y|\leq r(x)}\frac{|r(x)^{-\delta+k}\partial^{\gamma}u(x)-r(y)^{-\delta+k}\partial^{\gamma}u(y)|}{|x-y|^{\beta}}
		\] 
		is finite.
	\end{definition5}
	
	If $M$ is a $\mathcal{W}^{k,p}$ (resp. $\mathcal{C}^{k,\beta}$) manifold such that there is a compact set $Z$ for which $M\setminus Z$ consists of a finite, disjoint union $\bigsqcup_{l=1}^{N}E_{l}$ and for each index $l$ there is a diffeomorphism (of the appropriate level of regularity) $\Phi_{l}:E_{l}\to\mathbb{R}^{n}\setminus B_{l}$ (for some Euclidean ball $B_{l}$) then one can easily define the weighted spaces $\mathcal{W}^{k,p}_{\delta}$ (resp. $\mathcal{C}^{k,\beta}_{\delta}$). That is done, routinely, by choosing a finite atlas of $M$ consisting of the charts at infinity together with finitely many pre-compact charts and adding the $\mathcal{W}^{k,p}_{\delta}(\mathbb{R}^{n}\setminus B_{l})$ (resp. $\mathcal{C}^{k,\beta}_{\delta}(\mathbb{R}^{n}\setminus B_{l})$) norm on the former to the $\mathcal{W}^{k,p}$ (resp. $\mathcal{C}^{k,\beta}$) norm on the latter ones. The resulting spaces $\mathcal{W}^{k,p}_{\delta}$ (resp. $\mathcal{C}^{k,\beta}_{\delta}$) as well as their topology are not canonically defined, yet only depend on the choice of the diffeomorphisms $\Phi_{l}$ for $l=1,\ldots, N$.
	Such definition is easily extended to tensors of any type by following the very same pattern.

	\begin{definition6}\label{def:AFIDS}
		Given an integer $n\geq 3$, and real numbers $p>\frac{n-2}{2}, \ \beta\in(0,1)$ a complete manifold $(M,g)$ is called asymptotically flat of type $(p,\beta)$ if:
		\begin{enumerate}
			\item{there exists a compact set $Z\subset M$ (the \textsl{interior} of the manifold) such that $M\setminus Z$ consists of a disjoint union of \textsl{finitely many} ends, namely $M\setminus Z=\bigsqcup_{l=1}^{N}E_{l}$ and for each index $l$ there exists a smooth diffeomorphism $\Phi_{l}:E_{l}\to\R^{n}\setminus B_{l}$ for some open ball $B_{l}\subset \R^{n}$ containing the origin so that the pull-back metric $\left(\Phi_{l}^{-1}\right)^{\ast}g$ satisfies the following condition:
				\[
				\left(\left(\Phi_{l}^{-1}\right)^{\ast}g\right)_{ij}-\delta_{ij} \ \in \ \mathcal{C}^{2,\beta}_{-p}(\mathbb{R}^{n}\setminus B_{l})
				\]
			}
			\item{the \textsl{scalar curvature} $R$ is integrable, namely
				\[ R\ \in \ \mathcal{L}^{1}(M).
				\]
			}
		\end{enumerate}
	\end{definition6}
	
	From now onwards and throughout this article, we will assume to deal with asymptotically flat manifolds with only one end. This does not really cause any loss of generality for what concerns the proof of our rigidity results, the modifications to handle the case of multiple ends (of the ambient manifold) being of purely notational character.
	
	As first suggested in \cite{SY79} (but see also \cite{EHLS11}), for a number of purposes it is convenient to work, whenever possible, with asymptotically flat data that have a particularly simple description at infinity. 
	
	\begin{definition7}\label{def:HDS}
		Let $n\geq 3$, and let $(M,g)$ be an \textsl{asymptotically flat} manifold with one end. We say that $(M,g)$ is asymptotically Schwarzschildean if there exists diffeomorphisms (as in the Definition 6) as well as a function $h\in \mathcal{C}^{2,\beta}_{2-n}$ such that for $i,j=1,2,\ldots,n$
		\bd
		h(x)=1+a\left|x\right|^{2-n}
		\ed
		\bd
		g_{ij}=h^{\frac{4}{n-2}}\del_{ij}+O^{2,\beta}\left(\left|x\right|^{1-n}\right).
		\ed
	\end{definition7}
	
	We then recall the notion of ADM mass, which was introduced in \cite{ADM59} in the context of the Hamiltonian formulation of general relativity (see also \cite{Bar86} for a mathematical discussion of its well-posedness).
	
	\begin{definition8}\label{def:ADM}
		Given an asymptotically flat manifold $(M,g)$ with one end (so that the scalar curvature is integrable by assumption) one can define the ADM \textsl{mass} $\mathcal{M}$ to be
		\bd
		\mathcal{M}=\frac{1}{2\left(n-1\right)\omega_{n-1}}\lim_{r\to\infty}\int_{\left|x\right|=r}\sum_{i,j=1}^{n}\left(g_{ij,i}-g_{ii,j}\right)\frac{x^{j}}{\left|x\right|}\,d\mathscr{H}^{n-1}
		\ed
		where $\omega_{n-1}$ is the volume of the standard unit sphere in $\R^{n}$.
		%In case of multiple ends, these quantities can be defined by additive extension.
	\end{definition8}
	
	In 1979 Schoen and Yau proved the most fundamental property of this quantity, namely its positivity.
	
	\begin{theorem9}\cite{SY79, Wit81}\label{thm:PMT}
		Let $(M,g)$ be an asymptotically flat manifold of dimension $3\leq n< 8$ with one end and satisying the dominant energy condition. Then $\mathcal{M}\geq 0$ and equality holds if and only if $(M,g)$ is isometric to the Euclidean space $\left(\R^{n},\d\right)$.
	\end{theorem9}
	
	In fact, in case of multiple ends such inequality holds at the level of each end and a closer look at the proof of the equality case \cite{Sch84} shows that it is enough to have \textsl{one} end of null ADM mass to force the whole space to be globally isometric to $\left(\R^{n},\d\right)$.
	
	Our rigidity results need to make use of the physical meaning of the constant $a$ given in Definition 7. To that aim, we recall the following basic computation (the reader might check it, for example, in \cite{EHLS11}).
	
	\begin{lemma10}
		Let $(M,g)$ be an asymptotically flat manifold (with one end) having harmonic asymptotics and let $a$ be given by Definition 7. Then
		\bd
		\mathcal{M}=\frac{(n-2)}{2}a.
		\ed
	\end{lemma10}
	
	As a result, if we prove that an asymptotically Schwarzschildean manifold (of non-negative scalar curvature) has $a=0$ then it follows by the \textsl{Positive Mass Theorem} (Theorem 9) that the expansion of the metric has to be trivial to \textsl{all} orders (namely $g=\del$).
	
	The very same computation shows, more generally, the following: if $(M_{1},g_{1})$ is an asymptotically flat manifold (according to Definition 6) and $h=1+a|x|^{2-n}+O^{2,\beta}(|x|^{1-n})$ is $\mathcal{C}^{2,\beta}_{loc}$ then the ADM mass of $(M_{2},g_{2})$ where $g_{2}=h^{\frac{4}{n-2}}g_{1}$ and $n$ is the dimension of $M_{1}$ is given by
	\[
	\mathcal{M}_{2}=\mathcal{M}_{1}-\frac{1}{2\omega_{n-1}}\lim_{r\to\infty}\int_{|x|=r}h\nabla_{\nu_{1}}h\,d\mathscr{H}^{n-1}=\mathcal{M}_{1}+\frac{(n-2)}{2}a.
	\] 
	
	\
	\textbf{Notations.} We denote by $R$ (resp. $Ric(\cdot,\cdot)$) the scalar (resp. Ricci) curvature of $(M,g)$, by $R_{\Sigma}$ the scalar curvature of $\Sigma\hookrightarrow (M,g)$ and by $\nu$ (a choice of) its unit normal. This is globally well-defined if $\Sigma$ is two-sided (or, equivalently, orientable since $M$ is itself assumed to be orientable).  For Euclidean balls centered at the origin we omit explicit indication of the center and hence we write $B_{r}$ instead of $B_{r}(p)$. The latter notation is instead reserved to metric balls in the ambient manifold $(M,g)$. We let $C$ be a real constant which is allowed to vary from line to line, and we specify its functional dependence only when this is relevant or when ambiguity is likely to arise.   
	
	\section{Proof of the rigidity statements}
	
	The scope of this section is to give a detailed proof of Theorem 1 and Theorem 2. Our arguments turn out to be quite different in the case $n=3$ (treated in Section \ref{sec:rig3}) and $4\leq n<8$ (treated in Section \ref{sec:rig7}), even though the conceptual scheme is quite similiar and in both cases we make use of a deep result of L. Simon \cite{Sim83b, Sim85} concerning the uniqueness of tangent cones in the context of his study of isolated singularities of geometric variational problems. 
	Each of the two proofs is split into a sequence of lemmata which may be of independent interest, and in particular we remark that the union of Lemma 12, Lemma 13 and Lemma 14 gives a characterization, in the context of asymptotically flat spaces, of complete minimal surfaces of finite Morse index.
	
	\
	
	We shall start by introducing a viewpoint which will turn out to be very convenient. Given $\Sigma\hookrightarrow M$ a minimal hypersurface in an asymptotically flat manifold (Definition 6) we let $\Sigma^{i}$ be an \textsl{unbounded} connected component of $\Sigma\setminus Z$, where $Z$ is the core of $M$. Thus, there exists an end\footnote{In fact, $E$ is the only end of $M$ as we are assuming, for the sake of simplicity, to deal with asymptotically flat manifolds with only one end (as stated in Section 2).} $E$ of $M$ such that $\Sigma^{i}\hookrightarrow E$ and so (by means of the diffeomorphism $\Phi: E\to\mathbb{R}^{n}\setminus B$) we can consider a copy of $\Sigma^{i}$, which we will call $\Sigma^{i}_{0}$, as a submanifold with boundary in $\mathbb{R}^{n}\setminus B$ with the geometry induced by the ambient Euclidean metric. As a result, each $\Sigma^{i}_{0}$ is not minimal, but is a stationary point of a functional $\textbf{F}=\textbf{F}_{0}+\textbf{E}$ with $\textbf{F}_{0}$ the $(n-1)$-dimensional Hausdorff measure and $\textbf{E}$ an error term which decays at infinity with a rate that depends on the asymptotics of $g_{ij}-\delta_{ij}$ at infinity. Moreover, if $\Sigma$ is stable then each $\Sigma^{i}_{0}$ will be a stable stationary point for $\textbf{F}$. Finally, it is convenient to assume that the ball $B$ is centered at the origin (which of course we can always arrange, be means of a translation). We shall denote the ADM mass of the end $E$ by $\mathcal{M}$. 
	
	\subsection{The proof for $n=3$}\label{sec:rig3}

	As a preliminary and general remark, we shall remind the reader of a finiteness result (due to D. Fischer-Colbrie) concerning the structure at infinity of a complete minimal surface having finite Morse index in a 3-manifold of non-negative scalar curvature.
	
	\begin{lemma11}\label{lem:12}\cite{FC85}
		Let $(N,g)$ be a 3-manifold with non-negative scalar curvature, $\Gamma$ a complete oriented minimal surface in $N$. If $\Gamma$ has finite Morse index then there exists a compact set $\Omega$ and a smooth positive function $u$ that solves the equation $Lu=0$ on $\Gamma\setminus\Omega$. The metric $u^2g_{|\Gamma}$ is a complete metric on $\Gamma$ with non-negative Gaussian curvature outside $\Omega$. In particular, it follows that $\Gamma$ is conformally diffeomorphic to a complete Riemann surface with a finite number of points removed.
	\end{lemma11}
	
	From here onwards, let us get back to the asymptotically flat setting presented in the statement of Theorem \ref{thm:1}.
	
	\begin{lemma12}\label{lem:13}
		Let $\Sigma\hookrightarrow M^3$ be a complete minimal surface of finite Morse index. Given any $p_0\in M$ there exists $r_0>0$ such that for any $r>r_0$ the intersection of $\Sigma$ with $\partial B_r(p_0)$ in $M$ is transverse. Furthermore $\Sigma\setminus B_{r_0}(p_0)$ can be decomposed in a finite number of annular ends.		
	\end{lemma12}	
	
	By the word \textsl{annular}, we mean that each connected component of $\Sigma\setminus B_{r_{0}}(p_0)$ is diffeomorphic to an annulus $S^1\times [r_0,+\infty)$ (and that the same conclusion holds for every $r>r_0$).
	
	\
	
	In the setup described at the very beginning of this section, let us assume to fix an index $i$ and so we will denote, for simplicity of notation, $\Sigma^{i}_{0}$ by $\Sigma_{0}\hookrightarrow\mathbb{R}^{3}\setminus B$. Let us denote by $A_{0}$ the second fundamental form of $\Sigma_{0}$, by $H_{0}$ (resp. $K_{0}$) its mean (resp. Gaussian) curvature and by $\nu_{0}$ its Euclidean Gauss map. 
	Patently, the geodesic spheres in the asymptotically flat ambient manifold $(M,g)$ become arbitrarily close to coordinate spheres in each end for very large values of the radius: hence, it suffices to prove the assertions of Lemma 12 for $\Sigma_0$ in lieu of $\Sigma$.
	
	\begin{proof}
		By means of a variation on standard curvature estimates in \cite{Sch83b}, we know that there exists a constant $C>0$ such that (for some, hence for any, fixed point $p_0\in M$)
		\[ \sup_{p\in\Sigma\setminus Z} d_{g}(p,p_{0})|A(p)|\leq C, \ \ \sup_{p\in \Sigma} |A(p)|\leq C
		\]
		and thus, by virtue of the simple comparison result
		\[|A_{0}(x)-A(x)|\leq \frac{C}{|x|^{2}}\left(|x||A(x)|+1\right)
		\] 
		(where $\left\{x\right\}$ is a set of asymptotically flat coordinates in $E$) we have that
		\[ |A_{0}(x)|\leq \frac{C}{|x|}.
		\] 
		Similarly, since $\Sigma$ is assumed to be minimal one obtains that $|H_{0}(x)|\leq C|x|^{-2}$.

		These two facts imply that for any sequence $\lambda_{m}\searrow 0$ the rescaled surfaces $\lambda_{m}\Sigma_{0}$ converge (up to a subsequence, which we do not rename) to a \textsl{stable minimal lamination} $\mathcal{L}$ in $\mathbb{R}^{3}\setminus\left\{0\right\}$. The convergence happens, locally, in the sense of smooth graphs.
		
		Now, let $L$ be a leaf, namely a (maximal) connected component of $\mathcal{L}$. If $0\notin \overline{L}$ in $\mathbb{R}^{3}$, then $L=\overline{L}$ is a connected, stable minimal surface in the Euclidean space, hence a plane by \cite{FS80, dCP79}. However, the same conclusion is also true in the case when $0\in \overline{L}$ thanks to a removable singularity theorem obtained by Gulliver and Lawson \cite{GL86} (see also Meeks-Perez-Ros \cite{MPR13} and Colding-Minicozzi \cite{CM05}). As a result, every leaf of $\mathcal{L}$ is a flat plane in $\R^{3}$ and hence, since trivially any two leaves cannot intersect (due to the very definition of lamination) we conclude that, modulo an ambient isometry, $\mathcal{L}=\mathbb{R}^{2}\times Y$ for some $Y\subset\mathbb{R}$ closed.
		
		Because a plane in $\mathbb{R}^{3}$ is totally geodesic, it follows at once that we can upgrade our curvature estimate to $|A_{0}(x)|\leq o(1)|x|^{-1}$ as $x$ goes to infinity. This implies that for $r_{0}$ large enough the surface $\Sigma_{0}\setminus B_{r_{0}}$ (which, we recall, is assumed to be proper) satisfies $|A_{0}(x)|\leq C|x|^{-1}$ for some $C\in (0,1)$ and we claim this forces $\Sigma_{0}\setminus B_{r_{0}}$ to be annular (namely to consist of finitely many annular ends). 
		
		Indeed, let $f:\Sigma_{0}\to \mathbb{R}$ be the restriction of the Euclidean distance function $|\cdot|$ to the surface $\Sigma_0$. The critical points of $f$ occur at those points $x_{0}$ where the unit normal $\nu_{0}$ is parallel to the position vector $x$. A trivial computation shows that the Hessian at such a point is given by $\nabla^{2}f(x_{0})[v,v]=2(|v|^{2}-A_{0}(x_{0})[v,v]\delta(x_{0},\nu_{0}))$ for any vector $v\in T_{x_{0}}\Sigma_{0}\simeq\mathbb{R}^{2}$. We claim that indeed, under our assumptions $\nabla^{2}f(x_{0})[v,v]>0$ for all $v\in T_{x_{0}}\Sigma_{0}$ and, to check that, it is enough (by homogeneity) to reduce to the case when $v$ has unit (Euclidean) length. This is trivial, since $\nabla^{2}f(x_{0})[v,v]\geq 2(1-|A_{0}(x_{0})||x_{0}|)\geq 2(1-C)>0$.  
		This shows that all \textsl{interior} critical points of the function $f$ are strict local minima, which implies (since $\Sigma_{0}$ is connected) that $f$ does not have any interior critical points at all and so  $\Sigma_{0}$ is annular, as we had to prove. Indeed, if $f$ had an interior critical point (say $x_{\ast}$), we could easily obtain a second one (of saddle type) by means of a simple one-dimensional min-max scheme, namely looking at the space of paths connecting $x_{\ast}$ with a point on the boundary $\Sigma_0\cap \partial B_{r_0}$.
		At this stage, the fact that the number of ends in question is finite follows from Lemma 11.
	\end{proof}
	
	\textsl{By working one annular component at a time we can (and we shall) then assume, without renaming, that $\Sigma_{0}$ is a connected annulus, in the sense specified above.}
	
	\
	
	Let us explicitly remark that, by virtue of the argument above, for any sequence $\lambda_m\searrow 0$ the rescaled surfaces $\lambda_m \Sigma_0$ shall converge (in $\mathbb{R}^3\setminus\left\{0\right\}$) to a plane passing through the origin. The convergence is meant as single-sheeted, smooth graphical convergence. As a result, a straightforward blow-down argument ensures that $g(\nu_0,x/|x|)<1/3$ for $|x|$ large enough.
	
	\begin{lemma13}
		Let $\Sigma\hookrightarrow M^3$ be a complete, properly embedded minimal surface of finite Morse index. Then $\Sigma$ has quadratic area growth, namely there exists a positive constant $\theta^{\ast}$ such that the inequality $\mathscr{H}^2(\Sigma\cap B_r(p_0))\leq \theta^{\ast}r^2$ holds true for every $r>0$ and $p_0\in M$.
	\end{lemma13}	
	
	\begin{proof}
		Thanks to the finiteness of the number of ends of $\Sigma$ and the equivalence of the Riemannian metrics $g$ and $\delta$ on the end $E$ (by which we mean the existence of a constant $C>0$ such that $C^{-1}g\leq \delta\leq Cg$) the claim follows by virtue of the co-area formula after having shown the existence of a constant $C>0$ such that $\mathscr{H}^{1}(\Sigma_0\cap \partial B_r)\leq Cr$.
		
		If that were not the case, we could find an increasing sequence of radii $r_{m}\nearrow\infty$ (with $r_{m}>r_{0}$ for each $m$) such that the length of the embedded circle $\Sigma_0\cap \partial B_{r_{m}}$ is more than $r_m m$.  We could then consider the rescaled sequence gotten by taking $\lambda_{m}=r_{m}^{-1}$ and we should have, on the one hand
		\[ \mathscr{H}^1(\lambda_m\Sigma_0 \cap\partial B_1)\geq m
		\]
		(at least for $m$ large enough), while on the other we know that $\lambda_{m}\Sigma_0\cap \left(B_{3/2}\setminus B_{1/2}\right)$ is just a graph with small Lipschitz constant (by the argument we presented in the proof of Lemma 12), hence
		\[\mathscr{H}^1(\lambda_m\Sigma_0 \cap\partial B_1)\leq 3\pi .
		\]  
		thus reaching a contradiction.
	\end{proof}

	\begin{lemma14}
		Let $\Sigma\hookrightarrow M^3$ be a complete, properly embedded minimal surface of finite Morse index. Then $\Sigma$ is regular at infinity, namely it decomposes (outside a compact set) in a finite number of graphical components, each having an expansion of the form
				\[ u(x')=a\log|x'|+b+e(x'), \ \textrm{where} \ e(x')=O(|x'|^{-1+\varepsilon}) \ \ \forall \ \varepsilon>0
				\]
			for a suitably chosen set of asymptotically flat coordinates $\left\{x\right\}$.	
	\end{lemma14}	
	
	\begin{proof}
		Because of the quadratic area growth (gained in Lemma 13), the surface $\Sigma_{0}$ does admit a cone at infinity in the following sense. For any sequence $\lambda_{m}\searrow 0$ there exists a subsequence (which we do not rename) such that $\lambda_{m}\Sigma_{0}\rightharpoonup \Gamma$ for some stable minimal cone $\Gamma\subset\mathbb{R}^{3}$, hence a flat plane. The convergence happens in the sense of integral varifolds (see Chapter 4 of \cite{Sim83}), which is to say (in our special setting) that for any fixed continuous function $f$ with compact support in $\mathbb{R}^{3}\setminus \left\{0\right\}$ one has $\int_{\left\{\lambda_{m}x: \ x\in\Sigma_{0}\right\}}f\,d\mathscr{H}^{2}\longrightarrow\int_{\Gamma}f\,d\mathscr{H}^{2}$.
		Moreover, by virtue of the previous steps, we know that $\Sigma_{0}$ is annular, so that $\Gamma$ has multiplicity one. At this stage, we are then in position to apply Theorem 5.7 in \cite{Sim85} (see also the discussion given at pp. 269-270) which implies that $\Sigma_{0}$ is an outer-graph: there exists a function $u\in \mathcal{C}^{2}(\Pi\setminus B_{r_{0}};\mathbb{R})$ whose graph coincides with $\Sigma_{0}$ and moreover
		\[ |x'|^{-1}|u(x')|+|\nabla_{\Pi}u(x')| \ \rightarrow 0 \ \ \textrm{as} \ \ |x'|\to\infty.
		\]
		Here $\Pi$ is an hyperplane and we are adopting the notations presented in Section 1.
		To refine this information and get an asymptotic expansion for $u$ we need to preliminarily check that $\Sigma_0$ has finite total curvature, that is to say
		\[
		\int_{\Sigma_0}|A_0|^2\,d\mathscr{H}^2<\infty.
		\]
		This comes at once from the aforementioned comparison relation between $A$ and $A_0$ and thanks to the stability inequality for $\Sigma$ (in applying that we need to use the logarithmic cut-off trick, exploiting the quadratic area growth).
		
		This means that, by possibly taking a larger value of $r_{0}$ and choosing a suitable set of asymptotically flat coordinates $\left\{x\right\}$, the surface $\Sigma_{0}$ coincides with the graph of a smooth function $u:\Pi\to\mathbb{R}$ (for some plane $\Pi$) such that 
		\be\label{eq:graphrough}\lim_{|x'|\to\infty}\left|\nabla_{\Pi} u(x')\right|=0, \ \ \int_{\Pi\setminus B_{r_0}}\left|\nabla^2_{\Pi} u\right|^{2}\,d\mathscr{L}^{2}<\infty.
		\ee
		
		Moreover,  we can then exploit the rough information \eqref{eq:graphrough} to improve our decay estimate to $u(x')\leq C |x'|^{1-\alpha}$ (for $\alpha>0$) and hence use this in the (perturbed) minimal surface equation to get, by a standard elliptic bootstrap argument, that in fact
		\[ u(x')=a\log|x'|+b+e(x'), \ \textrm{where} \ e(x')=O(|x'|^{-1+\varepsilon}) \ \ \forall \ \varepsilon>0
		\]
		(the details are discussed in Appendix \ref{sec:add}). This shows that $\Sigma_{0}$ and hence $\Sigma$ is regular at infinity, which completes the proof.
	\end{proof}
	
	The following lemma relies on an interesting idea suggested by O. Chodosh and M. Eichmair which allows to shorten the argument we presented in the very first version of this article.
	
	\begin{lemma15}
		Let $\Sigma\hookrightarrow M^3$ be a complete minimal surface of finite Morse index. If $\mathcal{M}>0$, then the Gauss curvature of $\Sigma$ is negative outside a compact set.
	\end{lemma15}
	
	\begin{proof}
		It is enough to recall that 
		\[ Ric(\nu,\nu)+\frac{1}{2}|A|^{2}=\frac{1}{2}R-K
		\]
		and for our class of data (see Definition 7) we have
		\[R\leq C|x|^{-4}
		\]
		(because the Schwarzschild metric is itself scalar flat) and if $\mathcal{M}>0$ then
		\[ Ric(\nu,\nu)\geq C|x|^{-3}.
		\]  
		The latter assertion relies on the general formula for the Ricci tensor
		\[ Ric_{kl}=\frac{\mathcal{M}}{|x|^{3}}\left(1+\frac{\mathcal{M}}{2|x|}\right)^{-2}\left(\delta_{kl}-3\frac{x^{p}x^{q}}{|x|^{2}}\delta_{kp}\delta_{lq}\right)+O(|x|^{-4})
		\]
		and the fact that $g(\nu,x/|x|)<1/3$ for $|x|$ large enough, which has been remarked after the proof of Lemma 12.
	\end{proof}
	
	\begin{proof}[Proof of Theorem 1]
		
		Now, we are going to eploit all of this information about the behaviour of $\Sigma$ at infinity. Indeed, the stability inequality (with the rearrangement trick by Schoen-Yau) takes the form
		\[\frac{1}{2}\int_{\Sigma}\left(R+|A|^{2}\right)\phi^{2}\,d\mathscr{H}^{2}\leq \int_{\Sigma}|\nabla_{\Sigma}\phi|^{2}\,d\mathscr{H}^{2}+\int_{\Sigma}K\phi^{2}\,d\mathscr{H}^{2}
		\]
		so by means of the logarithmic cut-off trick
		\[\frac{1}{2}\int_{\Sigma}\left(R+|A|^{2}\right)\,d\mathscr{H}^{2}\leq \int_{\Sigma}K\,d\mathscr{H}^{2}
		\]
		and thanks to the Gauss-Bonnet theorem for open manifolds (see Shiohama \cite{Shi85} or \cite{Whi87}) we know that
		\[\int_{\Sigma}K\,d\mathscr{H}^{2}=2\pi(\chi(\Sigma)-P)
		\]
		(recall that $P\geq 1$ is the number of ends of $\Sigma$) hence
		\[ 0\leq\frac{1}{2}\int_{\Sigma}\left(R+|A|^{2}\right)\phi^{2}\,d\mathscr{H}^{2}\leq 2\pi(\chi(\Sigma)-P)
		\]
		which forces $\chi(\Sigma)=1$, $P=1$ and $\Sigma$ to be totally geodesic and with vanishing restriction of the ambient scalar curvature. At that stage, an argument by Fischer-Colbrie and Schoen \cite{FS80} gives that $\Sigma$ is intrinsically flat.
		
		Lastly, we recall Lemma 15, which ensures that if $\mathcal{M}>0$, then the Gauss curvature of $\Sigma$ is negative (at least far away from the core). Thus necessarily $\mathcal{M}=0$ and making use of Lemma 10 and the rigidity part of Theorem 9, we conclude that $(M,g)$ is the Euclidean space $\mathbb{R}^{3}$ which completes the proof.
	\end{proof}

	\subsection{The proof for $4\leq n<8$}\label{sec:rig7}
	
	We now move to the proof of Theorem 2, namely the higher dimensional counterpart of the previous one. It is convenient to recall here the notion of strong stability.
	
	\begin{definition16}
		Given $\alpha\in\mathbb{R}$ we set
		\[
		\mathcal{V}_{\alpha}(\Sigma)=\left\{\phi+\alpha \ \ | \ \ \phi\in \mathcal{W}^{1,2}_{\frac{3-n}{2}}(\Sigma)\right\}, 
		\]
		and we also define
		\[
		\mathcal{V}(\Sigma)=\bigcup_{\alpha\in\mathbb{R}}\mathcal{V}_{\alpha}(\Sigma).
		\]
		We say that a minimal hypersurface $\Sigma\hookrightarrow (M,g)$ is strongly stable if the stability inequality
		\[\int_{\Sigma}(Ric(\nu,\nu)+|A|^{2})\phi^{2}\,d\mathscr{H}^{n-1}\leq\int_{\Sigma}|\nabla_{\Sigma}\phi|^{2}\,d\mathscr{H}^{n-1}
		\]
		is true for any test function $\phi\in\mathcal{V}$.
	\end{definition16}
	
	When $\Sigma$ is known a priori to approach, at suitably good rate, an hyperplane along one of its ends the previous notion has a natural geometric interpretation: $\Sigma$ is strongly stable if it is stable with respect to all deformations that are essentially \textsl{vertical translations} near infinity.

	We work here with the very same notations defined at the beginning of the previous subsection and so let $\Sigma_{0}$ be (with slight abuse of notation) one unbounded connected component of $\Sigma\setminus Z$, and considered as a properly embedded submanifold of $\mathbb{R}^{n}\setminus B$ for some Euclidean ball $B$ centered at the origin. Of course $\partial\Sigma_{0}\subset \partial B$.
	
	\begin{lemma17}
		Let $\Sigma\hookrightarrow M^n$ be a complete, properly embedded strongly stable minimal hypersurface of polynomial volume growth. Given any $p_0\in M$ there exists $r_0>0$ such that for any $r>r_0$ the intersection of $\Sigma$ with $\partial B_r(p_0)$ in $M$ is transverse. 
		Such intersection consists of finitely many, say $P$,  smooth submanifolds $\Xi_{1},\ldots,\Xi_{P}$ of dimension $n-2$ and for every $r>r_0$ the set
		$\Sigma\setminus B_{r}(p_0)$ consists of $P$ ends and in fact we have a diffeomorphism $\Sigma\setminus B_{r}(p_0)\simeq \bigsqcup \Xi_{i}\times [r,+\infty)$.	
	\end{lemma17}	
	
	\begin{proof}
		
		First of all, let us observe that (if $\left\{x\right\}$ is a set of asymptotically flat coordinates in $\mathbb{R}^{n}$) then the integral
		\[ \int_{\Sigma_{0}}\frac{|x^{\perp}|^{2}}{|x|^{n+1}}\,d\mathscr{H}^{n-1}
		\]
		is finite (here $x^{\perp}=\delta(x,\nu_{0})$, the projection of the position vector onto the normal space of $\Sigma_{0}$ at the point in question). Such claim easily follows from the general monontonicity formula by Allard (see also Section 17 in \cite{Sim83}), true for any integral $k$-varifold $V$ of bounded mean curvature $H_{0}$
		\[
		\frac{\mu_{V}\left(B_{\rho}\left(\xi\right)\right)}{\rho^{k}}-\frac{\mu_{V}\left(B_{\sigma}\left(\xi\right)\right)}{\sigma^{k}}=\int_{B_{\rho}\left(\xi\right)}\frac{H_{0}}{k}\cdot \left(x-\xi\right)\left(\frac{1}{m(r)^{k}}-\frac{1}{\rho^{k}}\right)\,d\mu_{V}+\int_{B_{\rho}\left(\xi\right)\setminus B_{\sigma}\left(\xi\right)}\frac{\left|\partial^{\perp}r\right|^{2}}{r^{k}}\,d\mu_{V}
		\]
		(where $m(r)=\textrm{max}\left\{r,\sigma\right\}$): indeed, in our case the left-hand side is bounded because of the poylnomial volume growth assumption, and of course $\int_{\Sigma_{0}}\frac{|H_{0}|}{|x|^{n-2}}\,d\mathscr{H}^{n-1}$ is finite because $|H_{0}(x)|\leq C |x|^{-2}$ since $H=0$ identically, by minimality of $\Sigma$.
		
		We then claim that the previous integrability assertion can be turned into a pointwise decay estimate, in the sense that $|x|^{-1}|x^{\perp}|=o(1)$ as $|x|$ goes to infinity. To that aim, we argue as follows. 
		Suppose, by contradiction, that such statement were false: then we could find $0<\varepsilon<1$ and a sequence of points $\left(x_{i}\right)_{i\in\mathbb{N}}$ belonging to $\Sigma_{0}$ such that the following two conditions hold:
		\be\label{contrad}
		\begin{cases}
			\left|x_{i}\right|\nearrow \infty, \ \textrm{as} \ i\to\infty \\
			\frac{\left|x_{i}^{\perp}\right|}{\left|x_{i}\right|}\geq\varepsilon, \textrm{for all} \ i\in\mathbb{N}.
		\end{cases}
		\ee
		Now, observe that since $\Sigma$ is stable we know by making use of Theorem 3 in Section 6 of \cite{SS81} that there exists a constant $C>0$ such that
		\[ \sup_{p\in\Sigma\setminus Z} d_{g}(p,p_{0})|A(p)|\leq C, \ \ \sup_{p\in \Sigma} |A(p)|\leq C
		\]
		and, once again, by the comparison result $|A_{0}(x)-A(x)|\leq  \frac{C}{|x|^{2}}\left(|x||A(x)|+1\right)$ we conclude that $|A_{0}(x)|\leq \frac{C}{|x|}$.
		
		This immediately implies that $\left|\nabla_{\Sigma_{0}}\left(x\cdot\nu_{0}\right)\right|\leq C\left|x\right|\left|A_{0}(x)\right|\leq C$ for some constant $C\geq 1$.
		This gradient bound implies that if $r=\left|x\right|$ is large enough and $\left|x\cdot\nu_{0}\left(x\right)\right|\geq \e r$ then one should also have $\left|y\cdot\nu_{0}\left(y\right)\right|\geq\e r/2$ at all points $y\in\Sigma_{0}$ belonging to the following set:
		\bd
		\mathcal{K}_{x}=\left\{y\in\Sigma_{0} | \ \textrm{there exists a path} \ \gamma_{x,y}:\left[0,1\right]\to\Sigma_{0} \ \textrm{with} \ \gamma(0)=x, \ \gamma(1)=y, \ length(\gamma_{x,y})\leq \frac{\varepsilon r}{2C}\right\}
		\ed
		where $length(\gamma)$ denotes the length of the path $\gamma$ and $C$ is the constant defined above. We claim that in fact the set $\mathcal{K}_{x}$ contains the (extrinsic) ball of center $x$ and radius $\e r/\left(4C\right)$. This follows by a standard graphicality argument (see, for instance Lemma 2.4 in \cite{CM11}) again thanks to the pointwise decay assumption on the second fundamental form of $\Sigma_{0}$. If we apply this argument to each of the points $x_{i}$ keeping in mind their definition (see \eqref{contrad}) we get that 
		\bd
		\int_{\Sigma_{0}\cap \left\{\left|x\right|>r_{i}/2\right\}}\frac{\left|y\cdot \nu_{0}(y)\right|^{2}}{\left|y\right|^{n+1}}\,d\mathscr{H}^{n-1}\left(y\right)\geq \int_{\mathcal{K}_{x_{i}}}\frac{\left|y\cdot \nu_{0}(y)\right|^{2}}{\left|y\right|^{n+1}}\,d\mathscr{H}^{n-1}\left(y\right)
		\ed
		\bd
		\geq \int_{B_{\left(\frac{\e r_{i}}{4C}\right)}\left(x_{i}\right)}\frac{\left|y\cdot \nu_{0}(y)\right|^{2}}{\left|y\right|^{n+1}}\,d\mathscr{H}^{n-1}\left(y\right)\geq C\e^{n+1}
		\ed
		(for a suitable constant $C$) but on the other hand we already know, that it must be
		\bd
		\lim_{i\to\infty}\int_{\Sigma_{0}\cap \left\{\left|x\right|>r_{i}/2\right\}}\frac{\left|y\cdot \nu_{0}(y)\right|^{2}}{\left|y\right|^{n+1}}\,d\mathscr{H}^{n-1}\left(y\right)=0
		\ed
		and these two facts together give a contradiction. 
		
		Thanks to the statement we have just proved and the properness assumption, we can find a large number $r_{0}>0$ such that the submanifold $\Sigma_{0}$ meets all the spheres $\partial B_r$ \textsl{transversely} for $r\geq r_{0}$ and such intersection consists of finitely many submanifolds of dimension $n-2$. Furthermore, the \textsl{distance squared} function $r^{2}:\Sigma_{0}\cap B_{r_{0}}^{c}\to\R$ cannot have interior critical points for $r>r_{0}$  and therefore, by basic results in Morse theory and repeating the argument for each unbounded end of $\Sigma$ (one of finitely many, due to the polynomial volume growth of $\Sigma$) we conclude that
		$\Sigma\setminus B_{r}(p_0)$ consists of $P$ ends and $\Sigma\setminus B_{r}(p_0)\simeq \bigsqcup \Xi_{i}\times [r,+\infty)$, as we had claimed.
	\end{proof}
	
	\textsl{Without loss of generality, as we did above let us assume from now onwards that $\Sigma_{0}$ is renamed to be one of those ends.}
	
	\begin{lemma18}
		Let $\Sigma\hookrightarrow M^n$ be a complete, properly embedded strongly stable minimal hypersurface of polynomial volume growth. If $4\leq n<8$ then $\Sigma$ is regular at infinity, namely it decomposes (outside a compact set) in a finite number of graphical components, each having an expansion of the form
	 \[u(x')=a+b|x'|^{3-n}+e(x'), \ \textrm{where} \     e(x')=O(|x'|^{2-n+\varepsilon}) \ \ \forall \ \varepsilon>0
		\]
		for a suitably chosen set of asymptotically flat coordinates $\left\{x\right\}$.	
	\end{lemma18}	
	
	\begin{proof}
		Because of the polynomial volume growth assumption, the hypersurface $\Sigma_{0}$ does admit a cone at infinity in the sense recalled in the previous subsection. 
		Due to the fact that the ambient dimension is less than eight we then know by \cite{Sim68} that the cone $\Gamma$ has to be regular, since it is in fact an hyperplane $\Pi$. Moreover $\Gamma$ has multiplicity one. Hence we apply Theorem 5.7 in \cite{Sim85} (see also the discussion given at pp. 269-270) which ensures that $\Sigma_{0}$ is an outer-graph: there exists a function $u\in \mathcal{C}^{2}(\Pi\setminus B_{r_{0}};\mathbb{R})$ whose graph coincides with $\Sigma_{0}$ and moreover
		\[ |x'|^{-1}|u(x')|+|\nabla_{\Pi}u(x')| \ \rightarrow 0 \ \ \textrm{as} \ \ |x'|\to\infty
		\]
		At that stage, we make use of this information (together with the fact that $\int_{\Sigma}|A|^{2}\,d\mathscr{H}^{n-1}<\infty$, which is implied by the strong stability and the polynomial volume growth assumptions) to get for $u$ the desired expansion, in formal analogy with the $n=3$ case, as discussed in Appendix \ref{sec:add}. 
	\end{proof}

	To prove infinitesimal rigidity in dimension $k$ we will make use of the Positive Mass Theorem in dimension $k-1$. It is first convenient to state and prove the following simple lemma.
	
	\begin{lemma19}\label{lem:equal}Let $(M,g)$ and $\Sigma$ be as above. Suppose that there exists a function $\psi\in\mathcal{V}_{\alpha}(\Sigma)$ such that
		\[ \int_{\Sigma}\left(\left|\nabla_{\Sigma}\psi\right|^{2}+\frac{n-3}{4(n-2)}R_{\Sigma}\psi^{2}\right)\,d\mathscr{H}^{n-1}=0
		\]
		Then $\psi=\alpha$ for $\mathscr{H}^{n-1}$  a.e. $x\in\Sigma$.
	\end{lemma19}
	
	\begin{proof}If some function $\psi$ satisfies $\int_{\Sigma}\left(\left|\nabla_{\Sigma}\psi\right|^{2}+\frac{n-3}{4(n-2)}R_{\Sigma}\psi^{2}\right)\,d\mathscr{H}^{n-1}=0$, then obviously the integral $\int_{\Sigma}R_{\Sigma}\psi^{2}\,d\mathscr{H}^{n-1}\leq 0$ and therefore, because of the strong stability assumption, the fact that $c(k)=(k-2)/4(k-1)<1/2$ for any $k\geq 3$ and that the term $Q=\frac{1}{2}\left(R+|A|^{2}\right)$ is non-negative, we know that the following chain of inequalities holds:
		\[ \int_{\Sigma}\left|\nabla_{\Sigma}\psi\right|^{2}\,d\mathscr{H}^{n-1}\geq -\frac{1}{2}\int_{\Sigma}R_{\Sigma}\psi^{2}\,d\mathscr{H}^{n-1}\geq -c(n-1)\int_{\Sigma}R_{\Sigma}\psi^{2}\,d\mathscr{H}^{n-1}.
		\]
		As a result, necessarily $\int_{\Sigma}R_{\Sigma}\psi^{2}\,d\mathscr{H}^{n-1}=0$ and thus also $\int_{\Sigma}|\nabla_{\Sigma}\psi|^{2}\,d\mathscr{H}^{n-1}=0$. This forces $\psi$ to be constant $\mathscr{H}^{n-1}$ a.e. on $\Sigma$ and due to its behaviour at infinity (recall that we assumed $\psi\in\mathcal{V}_{\alpha}$) the conclusion follows.
	\end{proof}
	
	\begin{proof}[Proof of Theorem 2]
		We now derive infinitesimal rigidity for $\Sigma$. In view of Lemma 18 we know that $\Sigma$ itself can be regarded as an asymptotically flat manifold with an induced Riemannian metric given by
		\[ \overline{g}_{ij}=g\left(\frac{\partial U}{\partial x^{i}}, \frac{\partial U}{\partial x^{j}}\right)=h^{\frac{4}{n-2}}\left(\delta_{ij}+\frac{\partial u}{\partial x^{i}}\frac{\partial u}{\partial x^{j}}\right)+O(|x|^{-(n-1)}), \ \ \textrm{as} \ \ |x|\to\infty.
		\] 
		This is true at each of its ends ($u$ and $h$ are, respectively, the defining function of the end we are considering and the conformal factor of $E$).
		In order to deform $\overline{g}$ to a scalar flat metric on $\Sigma$ we need to prove that the conformal Laplacian $Q_{\Sigma}: \mathcal{W}^{1,2}_{\frac{3-n}{2}}(\Sigma)\to \mathcal{W}^{-1,2}_{\frac{-1-n}{2}}(\Sigma)$ is an isomorphism. First of all, Lemma 19 applied for $\alpha=0$ implies at once that $Q_{\Sigma}$ has to be injective. At that point, a standard application of the Sobolev inequality \cite{Bar86} gives that for any datum $\theta\in\mathcal{W}^{-1,2}_{\frac{-1-n}{2}}(\Sigma)$ the functional $\mathcal{F}_{\theta}$ given by
		\[ \varphi\mapsto \int_{\Sigma}\left(\left|\nabla_{\Sigma}\varphi\right|^{2}+\frac{n-3}{4(n-2)}\varphi^{2}-\theta\varphi\right)\,d\mathscr{H}^{n-1}
		\]
		is bounded from below and coercive and thus, following the direct method of the Calculus of Variations (where we exploit the Rellich compactness for these weighted spaces, as given in \cite{Bar86}), we conclude that it must have a critical point $\varphi_{0}\in \mathcal{V}_{0}(\Sigma)= \mathcal{W}^{1,2}_{\frac{3-n}{2}}(\Sigma)$. This completes the proof of the claim. When $\theta=-\frac{n-3}{2(n-2)}R_{\Sigma}$ we thus obtain a function $\chi$ with the property that $Q_{\Sigma}(1+\chi)=0$. If we let $\psi=1+\chi$ strong stability and integration by parts give
		\[ -c(n-1)\int_{\Sigma}R_{\Sigma}\psi^{2}\,d\mathscr{H}^{n-1}\leq 2c(n-1)\int_{\Sigma}\left|\nabla_{\Sigma}\psi\right|^{2}\,d\mathscr{H}^{n-1}\leq\int_{\Sigma}\left|\nabla_{\Sigma}\psi\right|^2\,d\mathscr{H}^{n-1}
		\]
		\[ =-c(n-1)\int_{\Sigma}R_{\Sigma}\psi^{2}\,d\mathscr{H}^{n-1}+\lim_{\sigma\to\infty}\int_{\partial B_{\sigma}}\psi\nabla_{\eta}\psi\,d\mathscr{H}^{n-2}.
		\]
		Since linear theory (see e. g. Meyers \cite{Mey63}) gives, for $\psi$, an expansion of the form
		\[ \psi(x')=1+\frac{2\mathcal{M}}{n-2}|x'|^{3-n}+O(|x'|^{2-n})
		\]
		we imediately see that previous inequality forces $\mathcal{M}\leq 0$. On the other hand, $\mathcal{M}$ represents the ADM mass of the asymptotically flat, scalar flat manifold $(\Sigma,\psi^{4/(n-3)}\overline{g})$ so that the \textsl{Positive Mass Theorem} (Theorem 9), in dimension $n-1$, gives $\mathcal{M}\geq 0$ and hence in fact $\mathcal{M}=0$.
		
		At that stage, we can re-consider the previous chain of inequalities and see at once that  we are in position to exploit Lemma 19 and conclude that $\psi=1$ identically on $\Sigma$. This means that $(\Sigma,\overline{g})$ had to be scalar flat.

		By performing computations similar to those we did above in the case $n=3$ we see that \textsl{if} $\mathcal{M}>0$ then, along the end $E$, one has $Ric(\nu,\nu)\geq C|x|^{n}$ while $|R|\leq C|x|^{n+1}$ so that the Gauss equation gives $R_{\Sigma}\leq -C|x|^{-n}$ as $x\to \infty$ (for $C>0$). This contradicts the conclusion stated in the previous paragraph, unless $(M,g)$ has mass zero, and the conclusion follows from Theorem 9.  
	\end{proof}

	\appendix
	
	\section{Improving decay from tangent cone uniqueness}\label{sec:add}
	
	We give here the details of an argument we mentioned both in Lemma 14 (auxiliary to the proof of Theorem 1) and in Lemma 18 (auxiliary to the proof of Theorem 2). In both cases, we could describe the hypersurface $\Sigma$ (or, more precisely, each end thereof) as the graph of a function $u\in\mathcal{C}^{2}\left(\Pi^{\ast}\simeq\mathbb{R}^{n-1}\setminus B;\mathbb{R}\right)$ with the properties that
	\[ \lim_{|x'|\to\infty}|\nabla_{\Pi} u(x')|=0, \ \ \int_{\Pi}|\nabla^2_{\Pi} u|^{2}\,d\mathscr{L}^{n-1}<\infty.
	\]
	(The function $u$ is only defined in the complement of an open ball in an hyperplane $\Pi$).
	We remark that the second condition is equivalent to the finiteness of the total curvature of $\Sigma$, since the gradient is bounded.
	
	\subsection{From H\"older decay of the gradient to optimal decay}
	
	In this subsection, we assume that $|\nabla^2_{\Pi} u|\leq C |x'|^{-1-\alpha}$ (in asymptotically flat coordinates $\left\{x\right\}$) for some $\alpha\in (0,1)$ and we indicate how an asymptotic expansion for $u$ can be obtained. First of all, let us observe that by our assumption we also get $|u|+|x'||\nabla_{\Pi} u|\leq C|x'|^{1-\alpha}$. Moreover, based on the fact that we are working with asymptotically flat data, all of the previous statements can be phrased in terms of Euclidean derivatives in asymptotically flat coordinates. A simple computation shows that the function $u$ solves a quasi-linear elliptic problem of the form
	
	\[ \sum_{i,j=1}^{n-1}\left(\delta_{ij}-\frac{u_{,i}u_{,j}}{1+|\partial u|^{2}}\right)u_{,ij}+2\left(\frac{n-1}{n-2}\right)\sqrt{1+|\partial u|^{2}}\frac{\partial h}{\partial \nu_{0}}+\mathcal{R}(x')=0
	\]
	
	where $\nu_{0}=\frac{\left(-\partial u, 1\right)}{\sqrt{1+|\partial u|^{2}}}$ is, as usual, the Euclidean unit normal and $\mathcal{R}$ is a remainder term such that
	
	\[\left|\mathcal{R}(x')\right|\leq C\left[\frac{|u|}{|x'|^{n+1}}\left(1+|\partial u|+|\partial u|^{2}\right)+\frac{1}{|x'|^{n}}\left(1+|\partial u|+|x'||\partial^{2}u|\right)\right].
	\]
	
	Based on the \textsl{a-priori} decay of $\partial u, \partial\partial u$ we can rewrite the equation in the much simpler form
	\[ \Delta u = f
	\] 
	where $|f(x')|\leq C|x'|^{\max\left\{-1-3\alpha, -n+1-\alpha, -n\right\}}$. Standard PDE theory guarantees that, given any $\varepsilon'>0$, we can find a solution $v$ of the problem
	\begin{equation*}
		\Delta v=f \ \textrm{on} \ \ \Pi^{\ast}=\mathbb{R}^{n-1}\setminus B 
	\end{equation*}
	satisfying the bound 
	\[
	|v(x')|+|x'||\partial v(x')|+|x'|^2|\partial^2 v(x')|\leq 
 C |x'|^{\max\left\{1-3\alpha,-n+3-\alpha,-n+2\right\}+\varepsilon'} \ 
	\] and thus necessarily $w=u-v$ is an harmonic function defined on the complement of a ball $B$ in $\mathbb{R}^{n-1}$. Since $w(x')=o(|x'|)$ as $|x|\to\infty$, if $n=3$ we conclude that $w$ grows at most logarithmically and has an expansion of the form
	\[ 
	w(x')=a_{-1}\log|x'|+a_{0}+\sum_{m\geq 1}\left(b_{m}\cos(m\theta)+b_{m}'\sin(m\theta)\right)|x'|^{-m}.
	\]
	Similarly, if $n\geq 4$ we conclude that $w(x')$ decays (modulo an additive constant) as prescribed by the inequality $w(x')\leq C|x'|^{3-n}$ and has an expansion. As a result, we obtain an improved bound for $u$ and we can exploit this information in the minimal surface equation solved by $u$.
	Iterating this argument finitely many times, we obtain that (for any given $\varepsilon>0$):
	\begin{itemize}
		\item{if $n=3$ grows at most logarithmically and $u(x')=a+b\log|x'|+O(|x'|^{-1+\varepsilon})$ as $|x'|\to\infty$;}
		\item{if $n\geq 4$ decays at a rate $|x'|^{3-n}$ and $u(x')=a+b|x'|^{3-n}+O(|x'|^{2-n+\varepsilon})$ as $|x'|\to\infty$.}
	\end{itemize}

	\subsection{Proving H\"older decay of the gradient}
	
	We then move to the preliminary part of the argument consisting in getting a pointwise decay estimate for $|\nabla u|$. By the De Giorgi Lemma (see, for instance, Theorem 5.3.1 in \cite{Mor66}) it is enough, to that aim, to prove an integral estimate of the form
	\[\int_{\Pi\setminus B_{\sigma}}|\nabla\nabla u|^{2}\,d\mathscr{L}^{n-1}\leq C\sigma^{-2\alpha}
	\]
	for some constant $C>0$ independent of $\sigma$. 
	Let us fix an index $1\leq k\leq n-1$ and differentiate in $x^{k}$ the equation solved by the function $u$: in turn $v_{k}=u_{,k}$ solves $T(v_{k})=\mathcal{R}_k$
	where
	\[ T(v_{k})=\partial_{x^{i}}(a^{ij}\partial_{x^{j}}v_{k})
	\]
	for $a^{ij}=\left(\delta^{ij}-{\nu_{0}}_{i}{\nu_{0}}_{j}\right)/\sqrt{1+|\partial u|^{2}}$ and $\left|\mathcal{R}_k(x')\right|\leq C|x'|^{-n}$ as $|x'|\to\infty$.
	From this equation we see that for any constant vector $\beta$ we have that $|\partial u-\beta|^{2}=\sum_{k}(v_{k}-\beta_{k})^{2}$ satisfies
	\[ T(\left|\partial u-\beta\right|^{2})=2\sum_{i,j,k}a^{ij}(\partial_{x^{i}}\partial_{x^{k}}u)(\partial_{x^{j}}\partial_{x^{k}}u)+2\sum_k \mathcal{R}_k (\partial_{x^k}u-\beta_k)\geq |\partial\partial u|^{2}-C|x'|^{-n}
	\]
	at least for $|x|$ large enough. That being said, for any large $\sigma$ we choose a cutoff function $\varphi$ which is one outside $B_{2\sigma}$ and zero inside $B_{\sigma}$ (this comes from the strong stability when $4\leq n<8$ and relies on the logarithmic cut-off trick when $n=3$ instead). We may multiply by $\varphi^{2}$ and integrate by parts using the integrability of $|\partial\partial u|^{2}$ to justify the result, thus obtaining
	\[\int_{B_{2\sigma}\setminus B_{\sigma}} -a^{ij}(\partial_{x^{i}}\varphi^{2})(\partial_{x^{j}}\left|\partial u-\beta\right|^{2})\,d\mathscr{L}^{n-1}\geq \int_{\Pi\setminus B_{\sigma}}\varphi^{2}|\partial\partial u|^{2}\,d\mathscr{L}^{n-1}-C\int_{\Pi\setminus B_{\sigma}}|x'|^{-n}.
	\]
	The standard manipulation (based on Young's inequality) and rearrangement give
	\[\int_{\Pi\setminus B_{\sigma}}\varphi^{2}|\partial\partial u|^{2}\,d\mathscr{L}^{n-1}\leq C\int_{B_{2\sigma}\setminus B_{\sigma}}|\partial\varphi|^{2}|\partial u-\beta|^{2}\,d\mathscr{L}^{n-1} +C\sigma^{-1}
	\]
	which of course implies
	\[\int_{\Pi\setminus B_{2\sigma}}|\partial\partial u|^{2}\,d\mathscr{L}^{n-1}\leq C\sigma^{-2}\int_{B_{2\sigma}\setminus B_{\sigma}}|\partial u-\beta|^{2}\,d\mathscr{L}^{n-1}+C\sigma^{-1}.
	\]
	Choosing the vector $\beta$ to be the average of the gradient $\partial u$ over the annulus and applying the Poincar\'e-Wirtinger inequality we obtain
	\[\int_{\Pi\setminus B_{2\sigma}}|\partial\partial u|^{2}\,d\mathscr{L}^{n-1}\leq C\int_{B_{2\sigma}\setminus B_{\sigma}}|\partial\partial u|^{2}\,d\mathscr{L}^{n-1}+C\sigma^{-1}.
	\]
	If we denote $J(\sigma)=\int_{\Pi\setminus B_{\sigma}}|\partial\partial u|^{2}\,d\mathscr{L}^{n-1}$, the previous inequality can be written in the form
	\[ J(2\sigma)\leq C(J(\sigma)-J(2\sigma))+C\sigma^{-1}
	\]
	or $J(2\sigma)\leq\theta J(\sigma)+\xi\sigma^{-1}$.
	At that stage, the claim is proved by means of a version \textsl{at infinity} of a general iteration lemma, as stated here.

	Let $\sigma\in\R, \ \sigma\geq\sigma_{0}$ and suppose we know that there exist constants $\theta\in\left(0,1\right)$, $\lambda>1$ and $\xi>0$ such that
	\[
	J(\l\sigma)\leq \theta J(\sigma)+\xi\sigma^{-1}, \ \forall \ \sigma\geq\sigma_{0}.
	\]
	It is convenient to set $\omega=-\log_{\l}\theta$ and let us notice that $\omega>0$, while we cannot say, at least a priori, whether $\omega\in\left(0,1\right]$ or instead $\omega>1$. Let us define $\tau=\sigma^{-1}$ and $G(\tau)=J(\frac{1}{\tau})$ for $\tau\in\left(0,\tau_{0}\right]$, where clearly $\tau_{0}=\sigma_{0}^{-1}$: our assumption turns into the equivalent form
	\[
	G\left(\frac{\tau}{\l}\right)\leq \l^{-\omega}G(\tau)+\xi\tau, \ \forall \ \tau\in\left(0,\tau_{0}\right].
	\]

	\begin{lemma}
		Let $G:\left(0,\tau_{0}\right]$ a non-decreasing function such that 
		\be\label{assumpt}
		G\left(\frac{\tau}{\l}\right)\leq \l^{-\omega}G(\tau)+\xi\tau, \forall \ \tau\in\left(0,\tau_{0}\right]
		\ee
		for some $\omega>0$ and $\xi>0$. Then, there exists a real constant $C=C\left(\omega,\xi,\l,\tau_{0}\right)$ such that 
		\[
		G(\tau)\leq C \tau^{\overline{\omega}}, \forall \ \tau\in\left(0,\tau_{0}\right]
		\]
		where we have set $\overline{\omega}=\min\left\{1,\omega\right\}$.
	\end{lemma}

	\begin{proof}
		Given $\tau\leq\tau_{0}$, let $Q\in\mathbb{N}$ be the only nonnegative integer such that
		\[
		\frac{\tau_{0}}{\l^{Q+1}}<\tau\leq\frac{\tau_{0}}{\l^{Q}}.
		\]
		Now if $Q\geq 1$, by our assumption \eqref{assumpt}, we know that
		\[
		G\left(\frac{\tau_{0}}{\l^{Q}}\right)\leq \l^{-\omega}G\left(\frac{\tau_{0}}{\l^{Q-1}}\right)+\xi\frac{\tau_{0}}{\l^{Q-1}}
		\]
		and then, by iteration, an elementary induction argument gives that
		\[
		G\left(\frac{\tau_{0}}{\l^{Q}}\right)\leq \l^{-Q\omega}G(\tau_{0})+\l^{\omega}\frac{\xi\tau_{0}}{\l^{Q}}\sum_{j=1}^{Q}\l^{\left(1-\omega\right)j}
		\]
		and therefore, since $G(\cdot)$ is nondecreasing, this implies
		\be\label{brute}
		G\left(\tau\right)\leq \l^{-Q\omega}G(\tau_{0})+\l^{\omega}\frac{\xi\tau_{0}}{\l^{Q}}\sum_{j=1}^{Q}\l^{\left(1-\omega\right)j}.
		\ee
		To proceed further, it is convenient to consider the two cases when $\omega>1$ or $\omega\in\left(0,1\right]$ separately.
		In the former, we immediately get from \eqref{brute}, by simply replacing the partial sum by the whole series (which is obviously \textsl{summable})
		\[
		G\left(\tau\right)\leq \l^{\omega}\left(\frac{\tau}{\tau_{0}}\right)^{\omega}G\left(\tau_{0}\right)+C(\omega,\xi,\l)\tau
		\]
		and hence
		\[
		G\left(\tau\right)\leq C(\omega,\xi,\l,\tau_{0})\tau, \ \forall \ \ \tau\in\left(0,\tau_{0}\right].
		\]
		In the latter case, it is enough to get an upper bound on such partial sum:
		\[
		\sum_{j=1}^{Q}\l^{\left(1-\omega\right)j}\leq \frac{\l^{1-\omega}}{\l^{1-\omega}-1}\left(\frac{\tau_{0}}{\tau}\right)^{1-\omega}
		\]
		and hence, from \eqref{brute} we imediately get
		\[
		G\left(\tau\right)\leq C(\omega,\xi,\l,\tau_{0})\tau^{\omega}, \forall \ \tau\in\left(0,\tau_{0}\right]
		\]
		which completes the proof.
	\end{proof}

	\textsl{Acknowledgments}. The author wishes to express his deepest gratitude to his PhD advisor Richard Schoen for his outstanding guidance and for his constant encouragement: his influence on this work is more than manifest. He would also like to thank Simon Brendle, Alessio Figalli, Andrea Malchiodi, Andr\'e Neves, Rafe Mazzeo, Brian White and especially Otis Chodosh and Michael Eichmair for several useful discussions and for their interest in this work.

	\bibliographystyle{plain}

\begin{thebibliography}{HKW}
		
		\setcounter{footnote}{0}
		
		
		\bibitem[Alm66]{Alm66}\textsc{F. J. Almgren Jr.}, \textit{Some interior regularity theorems for minimal surfaces and an extension of Bernstein's theorem}, Ann. of Math. \textbf{84} (1966), no. 2, 277-292.
		
		\bibitem[ADM59]{ADM59}\textsc{R. Arnowitt, S. Deser, C. W. Misner}, \textit{Dynamical structure and definition of energy in general relativity}, Phys. Rev. \textbf{2} (1959), no. 116, 1322-1330. 
		
		\bibitem[Bar86]{Bar86} \textsc{R. Bartnik}, \textit{The mass of an asymptotically flat manifold}, Comm. Pure Appl. Math. \textbf{39} (1986), no. 5, 661-693.
		
		\bibitem[Ber17]{Ber17}\textsc{S. N. Bernstein}, \textit{Sur une th\'eor\`eme de g\'eometrie et ses applications aux \'equations d\'eriv\'ees partielles du type elliptique}, Comm. Soc. Math. Kharkov \textbf{15} (1915-1917), 38-45 German translation in Bernstein, S. (1927), "\"Uber ein geometrisches Theorem und seine Anwendung auf die partiellen Differentialgleichungen vom elliptischen Typus", Mathematische Zeitschrift (Springer) \textbf{26}, 551-558.
		
		\bibitem[BdGG69]{BdGG69}\textsc{E. Bombieri, E. De Giorgi, E. Giusti}, \textit{Minimal cones and the Bernstein problem}, Invent. Math. \textbf{7} (1969), 243-268.
		
		\bibitem[CS14]{CS14} \textsc{A. Carlotto, R. Schoen}, \textit{Localizing solutions of the Einstein constraint equations}, Invent. Math. (\textsl{to appear}).
		
		\bibitem[CM05]{CM05}\textsc{T. Colding, W. Minicozzi II}, \textit{The space of embedded minimal surfaces of fixed genus in a 3-manifold V; Fixed genus},  Ann. of Math. (2) \textbf{181} (2015), no. 1, 1-153. 
		
		\bibitem[CM11]{CM11}\textsc{T. Colding, W. Minicozzi II}, \textit{A Course in Minimal Surfaces}, AMS Graduate studies in Mathematics, 2011.
		
		\bibitem[dG65]{dG65}\textsc{E. De Giorgi}, \textit{Una estensione del teorema di Bernstein}, Ann. Scuola Norm. Sup. Pisa \textbf{19} (1965), no. 3, 79-85. 
		
		\bibitem[dCP79]{dCP79}\textsc{M. do Carmo, C. K. Peng}, \textit{Stable complete minimal surfaces in R3 are planes}, Bull. Amer. Math. Soc. (N.S.) \textbf{1} (1979), no. 6, 903-906.
		
		\bibitem[EHLS11]{EHLS11}\textsc{M. Eichmair, L. H. Huang, D. A. Lee, R. Schoen}, \textit{The spacetime positive mass theorem in dimensions less than eight}, J. Eur. Math. Soc. \textbf{18} (2016), no. 1, 83-121.
		
		\bibitem[EM12]{EM12}\textsc{M. Eichmair, J. Metzger}, \textit{On large volume preserving stable CMC surfaces in initial data sets}, J. Differential Geom. \textbf{91} (2012), no. 1, 81-102.
		
		\bibitem[EM13]{EM13}\textsc{M. Eichmair, J. Metzger}, \textit{Large isoperimetric surfaces in initial data sets}, J. Differential Geom. \textbf{94} (2013), no. 1, 159-186.
		
		\bibitem[FC85]{FC85}\textsc{D. Fischer-Colbrie}, \textit{On complete minimal surfaces with finite Morse index in three-manifolds}, Invent. Math. \textbf{82} (1985), no. 1, 121-132.
		
		\bibitem[FS80]{FS80}\textsc{D. Fischer-Colbrie, R. Schoen}, \textit{The structure of complete stable minimal surfaces in 3-manifolds of nonnegative scalar curvature}, Comm. Pure Appl. Math. \textbf{33} (1980), no. 2, 199-211. 
		
		\bibitem[Fle62]{Fle62}\textsc{W. H. Fleming}, \textit{On the oriented Plateau problem}, Rend. Circ. Mat. Palermo \textbf{11} (1962), no. 2, 69-90. 
		
		\bibitem[GL86]{GL86}\textsc{R. Gulliver, H. Blaine Lawson Jr.}, \textit{The structure of stable minimal hypersurfaces near a singularity}, Geometric measure theory and the calculus of variations (Arcata, Calif., 1984),  Proc. Sympos. Pure Math., 44, Amer. Math. Soc., Providence, RI, 1986, 213-237. 
		
		\bibitem[MPR13]{MPR13}\textsc{W. H. Meeks III, J. Perez, A. Ros}, \textit{Local removable singularity theorems for minimal laminations}, J. Differential Geom. \textbf{103} (2016), no. 2, 319-362.
		
		\bibitem[Mey63]{Mey63} \textsc{N. Meyers}, \textit{An expansion about infinity for solutions of linear elliptic equations}, J. Math. Mech. \textbf{12} (1963) 247-264.
		
		
		\bibitem[Mor66]{Mor66}\textsc{C. B. Morrey}, \textit{Multiple integrals in the calculus of variations}, Die Grundlehren der mathematischen Wissenschaften, Band 130 Springer-Verlag New York, Inc., New York 1966 ix+506 pp.
		
		%\bibitem[Oss69]{Oss69}\textsc{R. Osserman}, \textit{A survey of minimal surfaces}, Math. Studies 25, Van Nostrand, New York, 1969.
		
		\bibitem[Sch83]{Sch83}\textsc{R. Schoen}, \textit{Uniqueness, symmetry, and embeddedness of minimal surfaces}, J. Differential Geom. \textbf{18} (1983), no. 4, 791-809 
		
		\bibitem[Sch83b]{Sch83b}\textsc{R. Schoen}, \textit{Estimates for stable minimal surfaces in Three Dimensional Manifolds} in \textit{Seminar on Minimal Subamanifolds}, Ann. of Math. Stud., vol. 103, Princeton Univ. Press, Princeton, NJ, 1983. 
		
		\bibitem[Sch84]{Sch84}\textsc{R. Schoen}, \textit{Conformal deformation of a Riemannian metric to constant scalar curvature}, J. Differential Geom. \textbf{20} (1984), no. 2, 479-495. 
		
		%\bibitem[Sch89]{Sch89}\textsc{R. Schoen}, \textit{Variational theory for the total scalar curvature functional for Riemannian metrics and related topics. Topics in calculus of variations (Montecatini Terme, 1987)}, 120-154,  Lecture Notes in Math., vol. 1365, Springer, Berlin,  1989. 
		
		\bibitem[SS81]{SS81} \textsc{R. Schoen, L. Simon}, \textit{Regularity of stable minimal hypersurfaces}, Comm. Pure Appl. Math. \textbf{34} (1981), no. 6, 741-797.
		
		\bibitem[SY79]{SY79}\textsc{R. Schoen, S. T. Yau}, \textit{On the proof of the positive mass conjecture in general relativity}, Comm. Math. Phys. \textbf{65} (1979), no. 1, 45-76.
		
		\bibitem[Shi85]{Shi85}\textsc{K. Shiohama}, \textit{Total curvatures and minimal areas of complete surfaces}, Proc. Amer. Math. Soc. \textbf{94} (1985), no. 2, 310-316.
		
		\bibitem[Sim83]{Sim83}\textsc{L. Simon}, \textit{Lectures on Geometric Measure Theory}, Centre for Mathematical Analysis (Australian National University), 1983.
		
		\bibitem[Sim83b]{Sim83b}\textsc{L. Simon}, \textit{Asymptotics for a class of nonlinear evolution equations, with applications to geometric problems}, Ann. of Math. (2) \textbf{118} (1983), no. 3, 525-571.
		
		\bibitem[Sim85]{Sim85}\textsc{L. Simon}, \textit{Isolated singularities of extrema of geometric variational problems. Harmonic mappings and minimal immersions}, Lecture Notes in Math. \textbf{1161} pp. 206-277, Springer, Berlin, 1985.
		
		\bibitem[Sim68]{Sim68}\textsc{J. Simons}, \textit{Minimal varieties in riemannian manifolds}, Ann. of Math.  \textbf{88} (1968), no. 2, 62-105. 
		
		\bibitem[Whi87]{Whi87}\textsc{B. White}, \textit{Complete surfaces of finite total curvature}, J. Differential Geom. \textbf{26} (1987), no. 2, 315-326. 
		
		\bibitem[Whi87b]{Whi87b}\textsc{B. White}, \textit{Curvature estimates and compactness theorems in 3-manifolds for surfaces that are stationary for parametric elliptic functionals}, Invent. Math. \textbf{88} (1987), no. 2, 243-256. 
		
		\bibitem[Wit81]{Wit81}\textsc{E. Witten}, \textit{A new proof of the positive energy theorem}, Comm. Math. Phys. \textbf{80} (1981), no. 3, 381-402. 
		
	\end{thebibliography}

\end{document}